\documentclass[12pt,reqno]{amsart} 
\usepackage{amssymb} 
\usepackage{enumerate}
\usepackage{mathrsfs}
\usepackage[notref,notcite]{showkeys}
\usepackage[all]{xy}
\usepackage{bm}
\usepackage{dsfont}
\usepackage{rotating}
\usepackage{bold-extra}

\usepackage[usenames]{color}
\usepackage[pdftex]{hyperref}

\newlength{\short}
\newlength{\shorter}
\definecolor{darkgreen}{rgb}{0,0.5,0}
\definecolor{bluegreen}{rgb}{0,0.2,0.8}
\definecolor{darkred}{rgb}{0.8,0,0}
\definecolor{newercolor}{rgb}{0.2,0,1}
\definecolor{darkyellow}{rgb}{0.7,0.7,0}
\definecolor{orange}{rgb}{0.9,0.4,0}

\newcommand{\mynote}[1]{\noindent{\color{bluegreen}\textup{\boldd{[#1]}}}}
\let\Gamma=\varGamma

\newcommand{\4}[1]{\overline{#1}}   %%{\widebar{#1}}

\newcommand{\too}{\longrightarrow}

\newcommand{\boldd}[1]{{\mathversion{bold}\textbf{#1}}}

\newcommand{\lie}[3]{\def\test{#2}\def\tst{G}\ifx\test\tst{{}^{#1}#2_{#3}}
\else{{}^{#1}\!#2_{#3}}\fi}

\let\oldcirc=\circ
\renewcommand{\circ}{\mathchoice
    {\mathbin{\scriptstyle\oldcirc}}{\mathbin{\scriptstyle\oldcirc}}
    {\mathbin{\scriptscriptstyle\oldcirc}}
    {\mathbin{\scriptscriptstyle\oldcirc}}}

\SelectTips{cm}{10} \UseTips   %% use computer modern tips in xy

\mathcode`\:="003A
\mathchardef\cdot="0201
\def\beq#1\eeq{\begin{equation*}#1\end{equation*}}
\def\beqq#1\eeqq{\begin{equation}#1\end{equation}}

\let\emptyset=\varnothing
\renewcommand{\:}{\colon}   %% as in f:X-->Y
\newcommand{\longline}{\bigskip\centerline{\hbox to 5cm{\hrulefill}}\bigskip}

\newcommand{\mxfoura}[8]{\left(\begin{smallmatrix}#1&#2&#3&#4\\#5&#6&#7&#8\\}
\newcommand{\mxfourb}[8]{#1&#2&#3&#4\\#5&#6&#7&#8\end{smallmatrix}\right)}

\DeclareMathAlphabet\EuR{U}{eur}{m}{n}
\SetMathAlphabet\EuR{bold}{U}{eur}{b}{n}

\newcommand{\higherlim}[2]{\displaystyle\setbox1=\hbox{\rm lim}
	\setbox2=\hbox to \wd1{\leftarrowfill} \ht2=0pt \dp2=-1pt
	\setbox3=\hbox{$\scriptstyle{#1}$}
	\def\test{#1}\ifx\test\empty
	\mathop{\mathop{\vtop{\baselineskip=5pt\box1\box2}}}\nolimits^{#2}
	\else
	\ifdim\wd1<\wd3
	\mathop{\hphantom{^{#2}}\vtop{\baselineskip=5pt\box1\box2}^{#2}}_{#1}
	\else
	\mathop{\mathop{\vtop{\baselineskip=5pt\box1\box2}}_{#1}}%
	\nolimits^{#2}
	\fi\fi}
\newcommand{\higherlimm}[2]{\setbox1=\hbox{\rm lim}
	\setbox2=\hbox to \wd1{\leftarrowfill} \ht2=0pt \dp2=-1pt
	\mathop{\mathop{\vtop{\baselineskip=5pt\box1\box2}}}\limits_{#1}
	\nolimits^{#2}}

\newcounter{let} 
\loop\stepcounter{let}
\expandafter\edef\csname cal\alph{let}\endcsname%
{\noexpand\mathcal{\Alph{let}}}
\ifnum\thelet<26\repeat

\loop\stepcounter{let}
\expandafter\edef\csname scr\alph{let}\endcsname%
{\noexpand\mathscr{\Alph{let}}}
\ifnum\thelet<26\repeat

\newcommand{\tdef}[2][]{\expandafter\newcommand\csname#2\endcsname%
{#1\textup{#2}}}
\tdef{Iso}   \tdef{Aut}    \tdef{Out}    \tdef{Inn}    \tdef{Hom}   
\tdef{Ext}   \tdef{Irr}    \tdef{rad}    \tdef{Sym}    \tdef{ord}
\tdef{End}   \tdef{Inj}    \tdef{map}    \tdef{Ker}    \tdef{Ob}    
\tdef{Mor}   \tdef{Res}    \tdef{Spin}   \tdef{Id}     \tdef{wt}
\tdef{rk}    \tdef{conj}   \tdef{incl}   \tdef{proj}   \tdef{diag} 
\tdef{trf}   \tdef{Sol}     \tdef{cj}    \tdef{Outdiag} \tdef{lcm}
\tdef{free}  \tdef{supp}   \tdef{soc}    \tdef{Ind}    \tdef{Tr}
\tdef[_]{typ}   \tdef[^]{op}    \tdef[^]{ab}  \tdef{Ree}
\tdef{Mon}   \tdef{Perm}   \tdef{Herm}   \tdef{Aff}
\tdef[^]{Ia} \tdef[^]{Ib} \tdef[^]{II} \tdef[^]{I}

\newcommand{\fdef}[1]{\expandafter\newcommand\csname#1\endcsname%
{\mathfrak{#1}}}
\fdef{X}  \fdef{red}  \fdef{foc}  \fdef{hyp}  
\newcommand{\bb}{\mathfrak{b}}

\newcommand{\bbdef}[1]{\expandafter\newcommand% 
\csname#1\endcsname{\mathbb{#1}}}
\bbdef{C} \bbdef{F} \bbdef{R} \bbdef{Z} \bbdef{Q}

\newcommand{\itdef}[1]{\expandafter\newcommand\csname#1\endcsname%
{\textit{#1}}}
\itdef{PSL}  \itdef{PSU}  \itdef{SL}  \itdef{SU}  \itdef{GL} \itdef{GU}
\itdef{UT}   \itdef{Sp}   \itdef{PSp} \itdef{PGL} \itdef{GO} \itdef{Co}
\itdef{SO}   \itdef{SD}   \itdef{Th}  \itdef{He}  \itdef{Sz} \itdef{AGL}
\itdef{Suz}  \itdef{McL}  \itdef{Ly}  \itdef{HS}  \itdef{Ru} \itdef{Fi}
\itdef{BDI}

\newcommand{\gen}[1]{\langle{#1}\rangle}

\newcommand{\syl}[2]{\textup{Syl}_{#1}(#2)}

\newcommand{\autf}[1][]{\Aut_{\calf_{#1}}}

\newcommand{\homf}[1][]{\Hom_{\calf_{#1}}}
\newcommand{\isof}[1][]{\Iso_{\calf_{#1}}}

\newcommand{\longleft}[1]{\;{\leftarrow%
\count255=0 \loop \mathrel{\mkern-6mu}%
    \relbar\advance\count255 by1\ifnum\count255<#1\repeat}\;}
\newcommand{\longright}[1]{\;{\count255=0 \loop \relbar\mathrel{\mkern-6mu}%
    \advance\count255 by1\ifnum\count255<#1\repeat\rightarrow}\;}

\newcommand{\RIGHT}[3]{\mathrel{\mathop{\kern0pt\longright#1}
	\limits^{#2}_{#3}}}

\newcommand{\LEFT}[3]{\mathrel{\mathop{\kern0pt\longleft#1}\limits^{#2}_{#3}}
}
\newcommand{\longleftright}[1]{\;{\leftarrow\mathrel{\mkern-6mu}%
    \count255=0\loop\relbar\mathrel{\mkern-6mu}% 
    \advance\count255 by1\ifnum\count255<#1\repeat\rightarrow}\;} 
 
\newcommand{\onto}[1]{\;{\count255=0 \loop \relbar\joinrel
    \advance\count255 by1
    \ifnum\count255<#1 \repeat \twoheadrightarrow}\;}

\newcommand{\RLEFT}[3]{\mathrel{%
   \mathop{\vcenter{\baselineskip=0pt\hbox{$\kern0pt\longright#1$}%
   \hbox{$\kern0pt\longleft#1$}}}\limits^{#2}_{#3}}}

\numberwithin{table}{section}

\renewenvironment{enumerate}[1][]
{\begin{enumerat}[#1]\setlength{\itemsep}{6pt}}{\end{enumerat}}

\renewenvironment{itemize}
{\begin{itemiz}\setlength{\itemsep}{6pt}\setlength{\itemindent}{-15pt}}
{\end{itemiz}}

\newenvironment{enuma}{\begin{enumerate}[{\rm(a) }]}{\end{enumerate}}
\newtheorem{Thm}{Theorem}
\newtheorem{Prop}[Thm]{Proposition}
\newtheorem{Cor}[Thm]{Corollary}
\newtheorem{Lem}[Thm]{Lemma}
\newtheorem{Claim}[Thm]{Claim}
\newtheorem{Ass}[Thm]{Assumption}
\newtheorem{Conj}[Thm]{Conjecture}
\newtheorem{Not}[Thm]{Notation}
\newtheorem{Hyp}[Thm]{Hypotheses}
\newtheorem{DefNot}[Thm]{Definition-Notation}

\theoremstyle{definition}
\newtheorem{Defi}[Thm]{Definition} 
\newtheorem{Rmk}[Thm]{Remark}
\newtheorem{Ex}[Thm]{Example}

\theoremstyle{remark}

\newcommand{\AAA}[1]{\def\test{#1}\def\tst{2}\textit{\textbf{A}}%
\ifx\test\tst\else^{(1#1)}\fi}
\newcommand{\LL}[1]{\def\test{#1}\def\tst{2}\textit{\textbf{L}}%
\ifx\test\tst\else^{(1#1)}\fi}
\newcommand{\NN}[1]{\def\test{#1}\def\tst{2}\textit{\textbf{N}}%
\ifx\test\tst\else^{(1#1)}\fi}

\def\Qtrp[#1,#2,#3]{{\ll}{#1},#2,#3{\gg}}  
\def\trp[#1,#2,#3]{{[\![}#1,#2,#3{]\!]}}
\def\Trp[#1,#2,#3]{\left[\!\!\left[#1,#2,#3\right]\!\!\right]}

\def\Qpr[#1,#2]{[\![#1,#2]\!]}

\newcommand{\xxx}{\textbf{\textit{x}}}

\title[Benson-Solomon fusion systems: correction]{Construction of 2-local 
finite groups of a type studied by Solomon and Benson: Correction}

\author{Bob Oliver}
\address{Universit\'e Sorbonne Paris Nord, LAGA, UMR 7539 du CNRS, 
99, Av. J.-B. Cl\'ement, 93430 Villetaneuse, France.}
\email{bobol@math.univ-paris13.fr}
\thanks{B. Oliver is partially supported by UMR 7539 of the CNRS}


\begin{document}

\begin{abstract} 
We correct an error in Lemma 3.1 of my paper \cite{LO} coauthored with Ran 
Levi, and show that the change does not affect any of the other results in 
that paper. More precisely, as pointed out to us by Justin Lynd, there are 
two conjugacy classes of elementary abelian subgroups of rank $3$ in each 
of the Benson-Solomon fusion systems $\calf_\Sol(q)$, and not only one as 
claimed in \cite{LO}.
\end{abstract}

\maketitle

In the paper \cite{LO} by Ran Levi and this author, we proved, for an odd 
prime power $q$, that a certain fusion system $\calf_\Sol(q)$ containing 
the $2$-fusion system of $\Spin_7(q)$ is saturated, thus constructing the 
first examples of exotic fusion systems over finite $2$-groups. We then 
showed that the classifying spaces of these fusion systems are the 
``homotopy fixed point sets'' of certain self maps of the space $\BDI(4)$ 
constructed by Dwyer and Wilkerson \cite{DW}. All of these results had been 
predicted by Dave Benson \cite{Benson}, and so our paper mostly involved a 
confirmation of his predictions. Shortly after publishing \cite{LO}, Andy 
Chermak pointed out that there was an error in our construction of 
$\calf_\Sol(q)$, and this led to some major changes needed to fix it, 
presented in \cite{LO-corr}. 

Recently, Justin Lynd pointed out another error, one which he discovered 
during work with Ellen Henke and Assaf Libman \cite{HLL}, and which 
(fortunately) does not affect the rest of the paper. We had claimed in 
\cite[Lemma 3.1]{LO} that all elementary abelian subgroups of rank $3$ in 
$\calf_\Sol(q)$ are conjugate to each other, while Lynd and his 
collaborators \cite[Lemma 4.8]{HLL} showed (correctly) that there are in 
fact two classes of such subgroups. Thus their proof of the corrected 
statement will soon appear in print. We have decided to publish this 
correction anyhow: partly to have an acknowledgement of the error more 
closely linked to our original paper (as a warning to anyone else using our 
paper as reference), partly to give another proof closer to the arguments 
and terminology used in \cite{LO}, and especially to show why this 
change does not affect any of the other results in our paper. 

We first recall some of the notation and terminology used in \cite{LO} to 
describe elements and subgroups of $\Spin_7(q)$ when $q$ is an odd prime 
power. We first look at $\Spin_7(q)$ and some of its subgroups. As usual, 
the discriminant of an $\F_q$-vector space with nondegenerate quadratic 
form is the determinant of the form with respect to some basis (well 
defined modulo squares in $\F_q^\times$). 

\begin{DefNot}[{\cite[Definition A.7]{LO}}] \label{d:types}
Fix an odd prime power $q$, a $7$-dimensional $\F_q$-vector space $V$, and 
a nonsingular quadratic form $\bb$ on $V$ with square discriminant.  
Identify $\SO_7(q)=\SO(V,\bb)$ and $\Spin_7(q)=\Spin(V,\bb)$. An elementary 
abelian 2-subgroup of $\SO_7(q)$ or of $\Spin_7(q)$ will be called of 
\emph{type I} if the eigenspaces for its action on $V$ all have square 
discriminant (with respect to $\bb$), and of \emph{type II} otherwise. 

Fix $S(q)\in\syl2{\Spin_7(q)}$, and let $\cale_n$ be the set of all 
elementary abelian $2$-subgroups of $S(q)$ of rank $n$ such that $z\in E\le 
S(q)$. Let $\cale_n\I$ and $\cale_n\II$ be the subsets of $\cale_n$ 
consisting of those subgroups of types I and II, respectively. \\
\end{DefNot}

The conjugacy classes of elementary abelian 2-subgroups of $\Spin_7(q)$ 
containing its center, and their automizers, are well known. The following 
is one way to describe them.

\begin{Prop}[{\cite[Proposition A.8]{LO}}] \label{p:elem.abel.} 
Fix an odd prime power $q$. Then 
\begin{itemize} 

\item $\cale_2\II=\emptyset$;

\item each of the sets $\cale_2\I$, $\cale_3\I$, $\cale_3\II$, and 
$\cale_4\II$ consists of exactly one conjugacy class of subgroups; 

\item the set $\cale_4\I$ contains exactly two $\Spin_7(q)$-conjugacy 
classes, denoted here $\cale_4\Ia$ and $\cale_4\Ib$, and they are exchanged 
by a diagonal outer automorphism of the group; and 

\item $\cale_n=\emptyset$ for $n\ge5$. 

\end{itemize}
Furthermore, the following hold.
\begin{enuma}  

\item For all $E\in\cale_2\cup\cale_3\cup\cale_4\I$, $\Aut_{\Spin_7(q)}(E) 
=\{\alpha\in\Aut(E)\,|\,\alpha(z)=z\}$. 

\item For all $E\in\cale_4$, $C_{\Spin_7(q)}(E)=E$.  

\item If $E\in\cale_3$, then $C_{\Spin_7(q)}(E)=A\gen{x}$, where $A$ 
is abelian, $x^2=1$, and $\9xa=a^{-1}$ for all $a\in A$. If 
$E\in\cale_3\II$, then the Sylow 2-subgroups of $C_{\Spin_7(q)}(E)$ are 
elementary abelian of rank $4$. 

\end{enuma}
\end{Prop}

\begin{proof} See \cite[Proposition A.8]{LO}: point (a) and the information 
about conjugacy classes are stated in the table, and points (b,c) are 
points (a,d) in that proposition. 
\end{proof}

\iffalse
More precisely, if $E\in\cale_3$, then the group $A$ in Proposition 
\ref{p:elem.abel.}(c) is abelian and homocyclic of rank $3$ and exponent 
$q\pm1$. \mynote{Find reference or drop.}
\fi

It remains to describe $\Aut_{\Spin_7(q)}(E)$ for $E\in\cale_4\II$. This is 
done in the next proposition. Let $\4\F_q$ be the algebraic closure of 
$\F_q$, and let $\psi^q\in\Aut(\Spin_7(\4\F_q))$ be the Frobenius 
automorphism. In particular, $C_{\Spin_7(\4\F_q)}(\psi^q)=\Spin_7(q)$. 

\begin{Prop}[{\cite[Proposition A.9]{LO}}] \label{p:E4II}
Fix an odd prime power $q$, and let $z\in{}Z(\Spin_7(q))$ be the central 
involution. 
\begin{enuma}  

\item For each $U\in\cale_4$, there is an element $g\in\Spin_7(\4\F_q)$ 
such that $\9gU\in\cale_4\Ia$. Set 
	\[ \xxx(U) = g^{-1}\psi^q(g) \]
for such $g$; then $\xxx(U)\in U$ and is independent of the choice of $g$. 

\item We have $U\in\cale_4\Ia$ if and only if $\xxx(U)=1$, and 
$U\in\cale_4\Ib$ if and only if $\xxx(U)=z$. 
 
\item Assume $U\in\cale_4\II$, and set $X=\gen{z,\xxx(U)}$.  
Then $\rk(X)=2$, and
        \[ \Aut_{\Spin_7(q)}(U) = \bigl\{\alpha\in\Aut(U) \,\big|\, 
        \alpha|_X=\Id \bigr\}. \]
Of the seven eigenspaces for the action of $U$ on $V$, the four on which 
$\xxx(U)$ acts via $-\Id$ all have nonsquare discriminant, while the three 
on which $\xxx(U)$ acts as the identity all have square discriminant.

\end{enuma}
\end{Prop}

What we call $\xxx(E)$ here is denoted $x_\calc(E)$ in \cite{LO}. 

Note that the type of $E$ is defined for all elementary abelian subgroups: 
if $z\notin E$ then the type of $E$ is the same as that of $E\gen{z}$. 
However, the statement that there are exactly two $\Spin_7(q)$-classes of 
subgroups isomorphic to $E_8$, determined by type, applies only to those 
subgroups that contain $z$. (Clearly, a subgroup that doesn't contain $z$ 
cannot be $\Spin_7(q)$-conjugate to one that does.)

This was the source of the error in the proof of \cite[Lemma 3.1]{LO}. We 
worked with the ``two conjugacy classes of rank 3 subgroups'', but forgot 
that this applies only to those subgroups that contain $z$. In particular, 
since an $\calf_\Sol(q)$-isomorphism from subgroup in $\cale_4\I$ to one in 
$\cale_4\II$ cannot send $z$ to itself, it sends subgroups of rank $3$ 
containing $z$ to ones not containing $z$, and so the argument involving 
types I and II does not apply. 

As stated above, we let $\cale_4\Ia$ and $\cale_4\Ib$ denote the two 
$\Spin_7(q)$-conjugacy classes in $\cale_4\I$. More precisely, in terms of 
the notation used in \cite{LO}, $\cale_4\Ia$ denotes the class of the 
subgroup $E_*=\gen{z,z_1,\5A,\5B}$. But these details aren't needed in what 
follows: what is important is that it is the class $\cale_4\Ib$ that fuses 
with $\cale_4\II$ in the larger fusion system $\calf_\Sol(q)$ (as stated 
below in Lemma \ref{l:3.1}(c)).

By Proposition \ref{p:elem.abel.}(c), each subgroup $E\cong E_8$ in $S$ is 
contained in some $U\cong E_{16}$. (If $z\notin E$, then $E\gen{z}\cong 
E_{16}$.)

\begin{Lem} \label{l:I-II}
For $U\in\cale_4\Ib\cup\cale_4\II$ and $E\le U$ of rank $3$ containing $z$,
	\beqq
	\textup{$E$ has type I if $\xxx(U)\in E$ and has type II if 
	$\xxx(U)\notin E$.} \label{eq}
	\eeqq
\end{Lem}

\begin{proof} If $U\in\cale_4\Ib$, then $\xxx(U)=z$ by Proposition 
\ref{p:E4II}(b), and each subgroup of rank $3$ containing $z$ has type I by 
definition. So \eqref{eq} holds in this case, and we assume from now on 
that $U\in\cale_4\II$.

Let $X_1,X_2,X_3,Y_1,Y_2,Y_3,Y_4$ be the seven eigenspaces of 
$U$ (corresponding to the seven nontrivial characters 
$U/\gen{z}\too\{\pm1\}$), labeled so that $\xxx(U)$ is the identity on the 
$X_i$ and acts by $-\Id$ on the $Y_i$. By Proposition \ref{p:E4II}(c), each 
$X_i$ has square discriminant, while each $Y_i$ has nonsquare discriminant. 
If $\xxx(U)\in E$, then each of the three 2-dimensional eigenspaces for $E$ 
(corresponding to the nontrivial characters $E/\gen{z}\too\{\pm1\}$) is a 
sum of two of the $X_i$ or two of the $Y_i$, and hence has square 
discriminant. So $E$ has type I in this case. If $\xxx(U)\notin E$, then 
each of the three 2-dimensional eigenspaces for $E$ is a sum $X_i\oplus 
Y_j$ for some $i,j$, hence has nonsquare discriminant, so $E$ has type II. 
\end{proof}

Thus either all rank 3 subgroups are $\calf_\Sol(q)$-conjugate, or there are 
exactly two classes, of which one contains all subgroups of 
$U\in\cale_4\Ib\cup\cale_4\II$ that contain $\xxx(U)$ as well as other 
subgroups containing $z$ of type I, and the other contains the subgroups of 
$U\in\cale_4\Ib\cup\cale_4\II$ that don't contain $\xxx(U)$. 

\begin{Lem}[Corrected version of {\cite[Lemma 3.1]{LO}}] \label{l:3.1}
Set $\calf=\calf_\Sol(q)$. 
\begin{enuma} 

\item For each $r\le2$, there is a unique $\calf$-conjugacy class of 
elementary abelian subgroups $E \le S(q)$ of rank $r$. 

\item There are two $\calf$-conjugacy classes of rank $3$ elementary 
abelian subtroups $E\le S(q)$: one of the classes contains $\cale_3\I$ and 
the other contains $\cale_3\II$. 

\item There are two $\calf$-conjugacy classes of rank $4$ elementary 
abelian subgroups $E \le S(q)$: $\cale_4\Ia$ and 
$\cale_4\Ib\cup\cale_4\II$. If $U_1,U_2\in\cale_4\Ib\cup\cale_4\II$ and 
$\varphi\in\homf(U_1,U_2)$, then $\varphi(\xxx(U_1))=\xxx(U_2)$. 

\end{enuma}
Furthermore, $\Aut_\calf(E)=\Aut(E)$ for all elementary abelian 
subgroups $E \le S(q)$ except when $E\in\cale_4\Ib\cup\cale_4\II$, in 
which case
	\[ \autf(E) = \{\alpha\in\Aut(E) \,|\, 
	\alpha(\xxx(E))=\xxx(E) \}. \]
\end{Lem}

\begin{proof} \boldd{Case 1: $\rk(E)=2$ or $4$: } Except for the last 
statement in (c), this holds by \cite[Lemma 3.1]{LO}. (This part of the 
proof of the lemma is correct.) If $\varphi\in\homf(U_1,U_2)$ where 
$U_1,U_2\in\cale_4\Ib\cup\cale_4\II$, then $\varphi$ is an isomorphism, and 
conjugation by $\varphi$ sends $\autf(U_1)$ to $\autf(U_2)$. Hence 
$\varphi(\xxx(U_1))=\xxx(U_2)$, since $\autf(U_i)$ is the group of all 
automorphisms that send $\xxx(U_i)$ to itself.

\smallskip

\noindent\boldd{Case 2: $\rk(E)=3$: } Assume there is only one 
$\calf_\Sol(q)$-class of subgroups of rank $3$. For $E\cong E_8$ with 
$z\notin E$, $C_{S(q)}(E)=C_{S(q)}(E\gen{z})=E\gen{z}$. Since the 
centralizer in $S(q)$ of each member of $\cale_3$ contains a copy of 
$E_{16}$ by Proposition \ref{p:elem.abel.}(c), there are rank 3 subgroups 
that are fully centralized in $\calf$ and contain $z$. 

\iffalse
Since there are rank 3 subgroups with larger centralizer (see, e.g., 
Proposition \ref{p:elem.abel.}(d)), all fully centralized subgroups in 
this class must contain $z$. 
\fi

So fix $E_*\in\cale_3$ (thus $z\in E_*$) that is fully centralized. Let $E$ 
be a rank 3 subgroup of $U\in\cale_4\Ib\cup\cale_4\II$ that contains $z$ 
and has the opposite type (I or II). By the extension axiom for $\calf$ 
(\cite[Proposition I.2.5]{AKO}), each $\varphi\in\isof(E,E_*)$ extends to 
some $\4\varphi\in\isof(U,U_*)$, where $U_*\in 
U^\calf=\cale_4\Ib\cup\cale_4\II$. Set $x=\xxx(U)$ and $x_*=\xxx(U_*)$. 
Then $f(x)=x_*$ by Case 1, and either $x\in E$ and $x_*\in E_*$, or 
$x\notin E$ and $x_*\notin E_*$. Hence $E$ and $E_*$ have the same type by 
Lemma \ref{l:I-II}, contradicting our assumption, and finishing the proof 
of (b). 

It remains to prove that $\autf(E)=\Aut(E)$ for all $E\in\cale_3$. If $E<U$ 
for $U\in\cale_4\Ia$, then $E\in\cale_3\I$ by definition, and 
$\autf(U)=\Aut(U)$ by Case 1. If $E<U$ for $U\in\cale_4\II$ and $\xxx(U)\notin E$ 
(recall $\xxx(U)\ne z$ in this case), then $\autf(E)=\Aut(E)$ since 
$\autf(U)=\{\alpha\in\Aut(U)\,|\,\alpha(\xxx(U))=\xxx(U)\}$, and 
$E\in\cale_3\II$ by Lemma \ref{l:I-II}. Thus $\autf(E)=\Aut(E)$ for at 
least some members of $\cale_3\I$ and some members of $\cale_3\II$, and 
hence for all elementary abelian subgroups of $S(q)$ of rank $3$. 
\end{proof}

Thus for $U\in\cale_4\Ia$, all rank 3 subgroups of $U$ are 
$\calf_\Sol(q)$-conjugate to each other (whether or not they contain $z$). 
For $U\in\cale_4\Ib\cup\cale_4\II$, all rank 3 subgroups of $U$ that 
contain $\xxx(U)$ are conjugate to each other, and all those that do not 
contain $\xxx(U)$ are conjugate to each other. 

Recall that the two $\Spin_7(q)$-conjugacy classes $\cale_4\Ia$ and 
$\cale_4\Ib$ are exchanged by a diagonal automorphisms of $\Spin_7(q)$. It 
is the precise construction of $\calf=\calf_\Sol(q)$ that determines which 
of those two classes fuses with $\cale_4\II$ in $\calf_\Sol(q)$.

\smallskip

\noindent\textbf{Affect on the rest of \cite{LO}: } 
Lemma 3.1 in \cite{LO} is referred to only twice later in the paper. In the 
proof of Lemma 3.2 (at the bottom of p. 944), we only need to know that 
$\Aut_{\calf_\Sol(q)}(E)=\Aut(E)$ for $E\in\cale_3$ (in either class), and 
this is shown in Lemma \ref{l:3.1}. 

In the proof of Lemma 4.1 (at the top of p. 952), we set $E=\Omega_1(Z(P))$ 
for a certain $\calf$-centric subgroup $P\le S$ (where 
$\calf=\calf_\Sol(q)$), and claim that $E$ is $\calf$-conjugate to a member 
of $\cale_3\I$. If not, then $E$ is $\calf$-conjugate to a 
member of $\cale_3\II$, and either $z\notin E$ (hence 
$C_S(E)=E\gen{z}\cong E_{16}$), or $E\in\cale_3\II$ (hence 
$C_S(E)\cong E_{16}$ by Proposition \ref{p:elem.abel.}(c)). Since $P\le 
C_S(E)$ and $P$ is $\calf$-centric, this situation is impossible.


\bigskip

\begin{thebibliography}{BLO2}

\iffalse
\bibitem[AC]{AC} M. Aschbacher \& A. Chermak, A group-theoretic approach to 
a family of $2$-local finite groups constructed by Levi and Oliver, 
Annals of Math. 171 (2010), 881--978.
\fi

\bibitem[AKO]{AKO} M. Aschbacher, R. Kessar, \& B. Oliver, Fusion systems 
in algebra and topology, Cambridge Univ. Press (2011)

\bibitem[Be]{Benson} D. Benson, Cohomology of sporadic groups, finite loop 
spaces, and the Dickson invariants, Geometry and cohomology in group 
theory, London Math. Soc. Lecture notes ser. 252, Cambridge Univ. Press 
(1998), 10--23

\bibitem[DW]{DW} W. Dwyer \& C. Wilkerson, A new finite loop space at 
the prime two, J. Amer. Math. Soc. 6 (1993), 37--64

\bibitem[HLL]{HLL} E. Henke, A. Libman, \& J. Lynd, Punctured groups for 
exotic fusion systems, arXiv 2201.07160v2

\bibitem[LO]{LO} R. Levi \& B. Oliver, Construction of 2-local finite
groups of a type studied by Solomon and Benson, Geometry \& 
Topology 6 (2002), 917--990.

\bibitem[LO2]{LO-corr} R. Levi \& B. Oliver, Correction to: Construction of 
2-local finite groups of a type studied by Solomon and Benson, 
Geometry \& Topology 9 (2005), 2395--2415.

\end{thebibliography}
\end{document}